\documentclass[12pt]{amsart}
\usepackage{amsmath,amsthm,amsfonts,amssymb}
\usepackage[all,cmtip]{xy}
\usepackage{verbatim}
 \usepackage{latexsym}
\newtheorem{thm}{Theorem}
\newtheorem{prop}[thm]{Proposition}

\newtheorem{lemma}[thm]{Lemma}

\theoremstyle{definition}
\newtheorem{eg}[thm]{Example}

\theoremstyle{remark}

\newtheorem{case'}{Case}

\numberwithin{equation}{section}
\numberwithin{thm}{section}

\renewcommand{\baselinestretch}{1.15}

\renewcommand{\leq}{\leqslant}
\def\({\left(}
\def\){\right)}
\def\[{\left[}
\def\]{\right]}
\def\={\quad = \quad}
\def\+{\quad + \quad}

\def\R{{\mathbb{R}}}
\def\Z{{\mathbb{Z}}}

\def\C{{\mathbb{C}}}
\def\HH{{\mathbb{H}}}

\def\em{\endmatrix\right]}

\def \g{\gamma}

\def \t {\bold t}

\def \a {\alpha}

\def \l {\lambda}

\def \G {\Gamma}
\def \H {\mathcal H}
\def \K {\mathcal K}

\newcommand{\note}[2][\null]{%
  \marginpar{\renewcommand{\baselinestretch}{1}\vspace{-1em}\hrule\vspace{3pt}%
  \footnotesize\raggedright\textsf{#2\ifx#1\null\else\\\hfill--- 
  {\em #1}\fi}\vspace{1.5em}}%
}

\begin{document}

\title[Abstract Parseval wavelet frames]{The orthonormal dilation property for abstract Parseval wavelet frames}

 \author{Bradley Currey and Azita Mayeli}


\date{\today}

\maketitle 

\begin{abstract}
 In this work we introduce a class of discrete groups containing subgroups  of abstract translations and dilations, respectively. A variety of wavelet systems can appear as $\pi(\G)\psi$, where $\pi$ is a unitary representation of a wavelet group and $\G$ is the abstract pseudo-lattice $\G$. We prove a condition in order that a Parseval frame  $\pi(\G)\psi$ can be dilated to an orthonormal basis of the form $\tau(\G)\Psi$ where $\tau$ is a super-representation of $\pi$. For a subclass of groups that includes the case where the translation subgroup  is Heisenberg, we show that this condition always holds, and we cite familiar examples as applications. 

  \end{abstract}


\section{\bf Introduction and preliminaries}
 
Given a Parseval frame $\{\psi_\a\}$ in a Hilbert space $\H$, it is known that there is a Hilbert space $\K$ and an orthornomal basis $\{\Psi_\a\}$ for $\K$ such that $\H \subset \K$ and $\psi_\a = P_\H(\Psi_\a)$ where $P_\H$ is the orthogonal projection of $\K$ onto $\H$ \cite{HL00}. In this case it is said that $\{\Psi_\a\}$ is an orthonormal dilation of $\{\psi_\a\}$.  If $\{\psi_\a\}$ is of the form $\pi(G)\psi$ where $G$ is group and $\pi$ is a unitary representation of $G$, then it is also known \cite{HL00} that there is an orthonormal dilation of the form $\tau(G)\Psi$, where $\tau$ is a unitary representation of $G$ acting in $\K$ such that $\tau(g)|_{\H} = \pi(g)$ for all $g \in G$ and such that $P_\H(\Psi) = \psi$. An affine wavelet system is not of the form $\pi(G)\psi$, but there is nevertheless an underlying group structure: it can regarded as having the form $\pi(\G)\psi$, where $\G$ is a discrete {\it pseudo-lattice} in a group $G$. For the wavelet system $\{2^{/2}\psi(2^j \cdot -k) : j\in \Z, k\in \Z\}$ in $L^2(\R)$, one can take $G$ to be the connected Lie group of affine transformations of the line with $\pi$ the quasiregular representation induced from the dilation subgroup, or (as in \cite{D2}) one can take $G$ to be the Baumslag-Solitar group $BS(1,2) = \langle u,t : utu^{-1} = t^2\rangle$ with $\pi(u)$ and $\pi(t)$ the 2-dilation and unit translation, respectively. When a Parseval wavelet frame has such a structure, it is natural to ask if there is an orthonormal dilation with the same structure: more precisely, if $\{\psi_\a\} = \pi(\G)\psi$,  is there a unitary representation $\tau$ of $G$ acting in a Hilbert space $\K$ containing $\H$, and a vector $\Psi\in \K$, such that $\tau(g)|_{\H} = \pi(g)$ for all $g \in G$ and such that $P_\H(\Psi) = \psi$? In this case we say that $\pi(\G)\psi$ has the $G$-dilation property, and it is then natural to ask for an explicit description of a $G$-dilation of $\pi(\G)\psi$. For the 2-wavelet system on the line, it is shown in \cite{D2} that for the $G = BS(1,2)$, every system $\pi(\G)\psi$ has the $G$-dilation property, and an explicit description of $G$-dilations of is carried out for Shannon-type wavelets. \\
 
In this paper we introduce a natural and general class of groups $G$, for which a number of well-known function systems, including both affine wavelet systems and shearlet systems, can be viewed as systems of the form $\pi(\G)\psi$ where $\G$ is a pseudo-lattice in $G$. We generalize the methods of \cite{D2} to prove a sufficient condition on the group $G$ in order that every such system has the $G$-dilation property. We then describe two natural families of wavelet groups and prove that they satisfy this sufficient condition. As one example, we exhibit a natural group $G$ and representation $\pi$ such that a shearlet system is of the form $\pi(\G)\psi$ and has the $G$-dilation property.

For the remainder of this paper, all groups are automatically countable and discrete. By {\it representation} of a group $G$, we shall mean a homomorphism of $G$ into the group of unitary operators on some Hilbert space $\H$ that is continuous in the strong operator topology. Representations will be assumed to {\it faithful}, that is, one-to-one mappings. \\
 
Let $\G_0$ be  a countable discrete group and $\a: \G_0 \rightarrow \G_0$ a monomomorphism. Define
  $$G(\a,\G_0):=\langle u,\G_0:  u\g u^{-1}=\a(\g),~~\forall ~ \g\in\G_0 \rangle.$$
The subset $ \G = \G_1\G_0$ where $\G_1 = \{u^j :  j\in \Bbb Z\}$ will be called the {\it standard pseudo-lattice} in $G$. As an example, observe that if $\G_0 = \Z$ and $\a_2$ is the monomorphism of $\Z$ defined by  $\a(1) = 2$, then $G(\a,\Z) = BS(1,2)$.
  
  

In the following section we use positive-definite maps to obtain a sufficient condition on the group $G$ in order that every Parseval wavelet frame  $\pi(\G)\psi$ has the $G$-dilation property. Then in Section \ref{expamples} we prove that our condition holds for two families of groups $G(\a,\G_0)$, and describe three examples.

 \section{\bf The group dilation property.}\label{D section}
 
  A map 
   $K : X \times X \rightarrow \C$ is  called a {\it positive definite map} if 
 for all finite sequences $\{\g_1, \g_2, \dots , \g_k\} $ in $X$ and $\{\xi_1, \xi_2, \dots , \xi_k\} \subset \C$, 
$$
\sum_{1 \le i, j \le k} K(\g_i,\g_j) \xi_i\overline{\xi}_j \ge 0.
$$


\noindent
If $X=G$ is  a group, then  following \cite{D1}, we say that $K : G \times G \rightarrow \C$  is a {\it group positive definite map} if  $K$ is a positive definite map and $K(sx,sy) = K(x,y)$ holds for all $s, x$ and $y$ in $G$. By \cite[Theorem 2.8]{D1}, every group positive definite map has the form 
$$
K_{\rho,\eta}(x,y) = \langle\rho(x)\eta, \rho(y)\eta\rangle
$$
where $\rho$ is a representation of $G$ and $\eta$ is a cyclic vector for $\rho$. 

For the remainder of this section, we fix a group $G = G(\a,\G_0)$, with $\G = \G_1\G_0$, and write the element $u^j\gamma \in \G$ as $(j,\gamma)$. 
 Let $\rho$ be a representation of $\Gamma_0$; we say that a  representation $T$ is an {\it $\a$-root  of   $\rho$} if $T\circ \a=\rho$. 
In the following abstract version of \cite[Theorem 2.1]{D2}, we use this notion to formulate a sufficient condition in order that a positive definite map on $\G$ extends to a group positive definite map.

\begin{prop}\label{second-technical}
Suppose that every representation  of $\G_0$  has an $\a$-root.  Let $K : \G \times \G \rightarrow \C$ be a positive definite mapping such that for any $(j, \g) $ and $(j',\g') $ in $\G$,  and $\g_0\in \G_0$, the relations

 \begin{equation}\label{K relations}
 \begin{aligned}
 K((j+1,\g), (j'+1,\g'))&=  K((j,\g), (j',\g'))\\
 K((j,\a^{-j}(\g_0)\g), (j',\a^{-j'}(\g_0)\g'))&=  K((j,\g), (j',\g')), ~ ~ j\leq 0,
 \end{aligned}
  \end{equation}
  both hold. 
  Then
  $K$ is the restriction of a group positive definite map $K_{\tau,\psi}$. More explicitly, 
   there is a representation $\tau$ of $G$ acting in a  Hilbert space $\H$, and a vector $\psi\in \H$, such that   $\H=\overline{span}\{\tau(\G)\psi\}$  and 
  $$
  K((j,\g), (j',\g'))= \langle \tau(j,\g)\psi, \tau(j',\g')\psi \rangle. 
  $$


\end{prop}

\begin{proof} By a theorem attributed to Kolmogorov (see for example \cite{D1}),
we have a Hilbert space $\mathcal H$ and a mapping $v : \G \rightarrow \mathcal H$, such that $\text{span} \{v(j,\g) : ~(j,\g)\in\G\} $ is dense in $\H$,  and 
$$
K((j,\g),(j',\g'))  = \langle v(j,\g),v(j',\g')\rangle 
$$
holds for all $(j,\gamma)$ and $ (j',\gamma')$ belonging to $\G $. Define the operator  $D : \mathcal H \rightarrow \mathcal H$ by $D v(j,\g) = v(j+1,\g)$ and by extending to all of $\H$ by linearity and density as usual. The first of the relations (\ref{K relations}) shows that $D$ is unitary. For each $n = -1, 0, 1, 2, \dots$, set 
$$
\mathcal H_n = \overline{\text{span}} \{v(j,\g)  : ~(j,\g)\in\G, j\le n\}.
$$
Note that $D \mathcal H_n = \H_{n+1}$ and $\H_n\subset \H_{n+1}$. Set $\K_n= \H_n \ominus \H_{n-1}$, $n \ge 0$. For $\gamma_0 \in \Gamma_0$, define the operator $T_0(\gamma_0)$ on $\H_0$ by 
$$
T_0(\g_0)\bigl( v(j,\g)\bigr) = v(j, \a^{-j}(\g_0)\g)
$$
and again extending to all of $\H_0$; the second relation in (\ref{K relations}) shows that $\gamma \mapsto T_0(\gamma)$ is a (unitary) representation of $\G_0$. Since the subspace $\K_0$ is invariant under $T_0$,  we can define the representation $\rho_1$ of $\G_0$ acting in $\K_1$ by $\rho_1(\g) = DT_0(\g)D^{-1}$. Now by our hypothesis, $\rho_1$ has an $\a$-root $T_1$: $T_1$ acts in $\K_1$ and satisfies $T_1 \circ \a = \rho_1$. Now the representation $\g \mapsto \rho_2(\g) = DT_1(\g)D^{-1}$ of $\G_0$ acting in $\K_2$ has an $\a$-root $T_2$ acting in $\K_2$. Continuing in this way, we obtain, for each positive integer $n$, a representation $T_n$ of $\G_0$ acting in $\K_n$, so that 
$$
T_n \circ \a = DT_{n-1}D^{-1}.
$$
(Again in the preceding, $T_0$ is restricted to $\K_0$.) Now write 
$$
\H = \H_0 \bigoplus\left( \oplus_{n \ge 1} \K_n\right)
$$
and define the representation $T$ of $\G_0$ by $T =T_0  \bigoplus \oplus_{n \ge 1} T_n$.
 
 Next we must verify the relation $DT(\gamma)D^{-1} = T \bigl(\a(\gamma)\bigr)$.
Fix $\g_0 \in \G_0$; for $v(j,\g)$ with $j \le 0$,
$$
\begin{aligned}
\bigl( DT_0(\g_0) D^{-1}\bigr) \bigl(v(j,\g)\bigr) &= \bigl(D T_0(\g_0)\bigr) \bigl( v(j-1,\g)\bigr) \\
&= D\bigl( v(j-1, \a^{-j+1}(\g_0)\g)\bigr) \\
&=T_0\bigl( \a(\g_0)\bigr)\bigl(v(j,\g)\bigr)
\end{aligned}
$$
and hence the relation $DT_0(\g) D^{-1} = T_0 \bigl(\a(\gamma)\bigr)$ holds on $\H_0$. Now for $v \in \H$, write $v = \sum_{n \ge 0} v_n$. We have $D T(\g)D^{-1}v_{0} = T \bigl(\a(\g)\bigr) v_{0}$ and for $n \ge 1$, $D T(\g)D^{-1}v_n = D T_{n-1}(\g)D^{-1}v_n = T_n\bigl(\a(\g)\bigr)v_n$ so 
$$
D T(\g)D^{-1}v =  \sum_{n\ge 0} D T(\g)D^{-1}v_n = \sum_{n\ge 0} T_n\bigl(\a(\g)\bigr)v_n =  T\bigl(\a(\g)\bigr).
$$
It follows that the mapping $\tau$ defined by $\tau(u) = D$ and $\tau(\g) = T(\g)$ is a representation of $G$. 

 Finally, take  $\psi=v(0,0)$. Then  $$v(j,\gamma)= D^j v(0,\g)= D^j T(\g)v(0,0)= D^j T(\g)\psi$$ 
so $\psi$ is cyclic for $\tau$. Hence the group positive definite map defined for all $x,y \in G$ by $K_{\tau,\psi}(x,y)=\langle \tau(x)\psi,\tau(y)\psi\rangle$ is an extension of $K$. 
\end{proof}

We combine the preceding with general results also from  \cite{D2} to obtain our condition for  the $G$-dilation property.

  \begin{thm}\label{main theorem} Suppose that every representation  of $\G_0$  has an $\a$-root., and let $\pi$ be any representation of $G(\a,\G_0)$. 
Then every Parseval wavelet frame  $\pi(\G)\psi$ has  the $G$-dilation property. 
 
 
 \end{thm}
 
\begin{proof}
Let $\G = \G_1\G_0 \subset G$ as above and recall that we write $u^j\gamma = (j,\gamma)$. Define 
$$
K((j,\g),(j',\g')) = \delta_{j,j'}\delta_{\g,\g'} - \langle \pi(j,\g)\psi,\pi(j',\g')\psi\rangle.
$$
Observe that  $\delta_{j+1,j'+1} = \delta_{j,j'}$ and 
$$
 \delta_{j,j'}\delta_{(a^{-j}\g_0)\g,(a^{-j'}\g_0)\g'} =  \delta_{j,j'}\delta_{\g,\g'},
$$
and that in the group $G$, $u^{-j}\g_0u^j = a^{-j}\g_0$ holds for $j$, $\g_0\in\G_0$. 
Hence
$$
\begin{aligned}
K((j+1,\g),(j'+1,\g')) &= \delta_{j+1,j'+1}\delta_{\g,\g'} - \langle \pi(j+1,\g)\psi,\pi(j'+1,\g')\psi\rangle\\
&=  \delta_{j,j'}\delta_{\g,\g'} - \langle D\pi(j,\g)\psi,D\pi(j',\g')\psi\rangle\\
&= K((j,\g),j',\g'))
\end{aligned}
$$
and for $j \le 0$, 
$$
\begin{aligned}
K((j,(\a^{-j}\g_0)\g), &(j',(\a^{-j}\g_0)\g'))\\
&= \delta_{j,j'}\delta_{(\a^{-j}(\g_0)\g,\a^{-j'}(\g_0)\g'} - \langle \pi(j  \a^{-j}(\g_0)\g)\psi,\pi(j' \a^{-j'}(\g_0)\g')\psi\rangle\\
&= \delta_{j,j'}\delta_{\g,\g'} -  \langle \pi(\g_0j\g)\psi,\pi(\g_0j'\g')\psi\rangle \\
&=  \delta_{j,j'}\delta_{\g,\g'} -  \langle \pi(\g_0)\pi(j\g)\psi,\pi(\g_0)\pi(j'\g')\psi\rangle\\
&=  \delta_{j,j'}\delta_{\g,\g'} -  \langle \pi(j\g)\psi,\pi(j'\g')\psi\rangle\\
&= K((j,\g),(j',\g')). 
\end{aligned}
$$
  
  The calculations show that the map $K$ satisfies the both conditions (\ref{K relations}). By Proposition \ref{second-technical} we conclude that $K$ is a positive definite map and hence there exists a representation $\tau$ of $G$ with Hilbert space $\K$ and $\eta\in \K$ such that $K=K_{\tau,\eta}$ on $\G\times \G$. Then \cite[Lemma 2.5, proof of Theorem 2.6]{D2} $\pi \oplus \tau$ is a super-representation of $\pi$ (acting in $\H \oplus \K$) for which  $\tilde\psi = \psi\oplus\eta$ is 
  a $G$-dilation vector for $\psi$ and  $\tilde\pi(x)\psi= \pi(x)\psi$. \end{proof}

Observe that in the case of $BS(1,2) = G(\a_2,\Z)$, the fact that every representation of $\G_0$ has an $\a$-root is a simple consequence of the Borel functional calculus: for every unitary operator $T$ on a Hilbert space $\H$, there is a unitary operator $S$ such that $S^2 = T$. However, in general it seems difficult to prove that a pair $(\a,\G_0)$ has the property that every representation of $\G_0$ has an $\a$-root.
In the following section we describe two families of groups $G(\a,\G_0)$ for which this property does in fact hold.


\section{\bf Examples}\label{expamples}

We begin with the case where $\G_0$ is a finitely-generated and abelian group.

 \begin{eg}(${\bf A}$-wavelet system)\label{A}
\rm  Let $\G_0$ be the free abelian group generated by $ t_1, t_2,  \cdots, t_n$, 
and let $\a(t_j) = t_1^{a_{1j}} t_2^{a_{2j}} \cdots t_n^{a_{nj}}$ where ${\bf A} = [a_{i,j}] \in GL(n,\Z)$. 

We claim that every representation of $\G_0$ has an $\a$-root. Let $\rho$ be any representation of $\G_0$, and write ${\bf A}^{-1}=[b_{i,j}]$. Since the $b_{i,j}$ are rational, the Borel functional calculus obtains operators $V_{i,j}, 1 \le i,j \le n$ such that $V_{i,j} =  \rho(t_1)^{b_{i,j}}$. 
Define $T(t_j), 1 \le j \le n$ by 
$$
T(t_j) =V_{1,j}V_{2,j}\cdots V_{n,j}.
$$
An easy computation shows that $T\circ\a = \rho$.


\end{eg}

Next we consider wavelet groups where the subgroup $\G_0$ is nilpotent, but not abelian. Nearest to the abelian case is the case where $\G_0$ is Heisenberg: let $\G_0 = \langle t_1, t_2, t_3\rangle$ with relations $ t_3t_2 = t_1 t_2t_3, t_1t_2 = t_2t_1, t_1t_3 = t_3t_1$. Then $\G_0$ is isomorphic with the discrete Heisenberg group
$$
\HH = \left\{\left[\begin{matrix} 1 & k &m \\ 0 & 1 &l \\ 0 & 0 & 1\end{matrix} \right] : k,l, \text{and } m \text{ are integers} \right\}
$$
via the map $t_1 \mapsto t_1^m, t_2 \mapsto t_2^l, t_3 \mapsto t_3^k$, and we identify $\G_0 = \HH$. For any positive numbers $a$ and $b$, the mapping $\a$ defined by $\a(t_3) = t_3^a, \a(t_2) = t_2^b, \a(t_1) = t_1^{ab}$ is a monomorphism of $\HH$. 
When $\a$ is of the form above, we use the notation $G(\a,\HH) = G(a,b,\HH)$. The following lemma shows that, at least where $a$ and $b$ are integers, $G(a,b,\HH)$ has the $\a$-root property. 
 



\begin{lemma}\label{first-technical-general version} Let $A, B$, and $C$ be unitary operators on a Hilbert space $\H$ satisfying $AB =CBA, ~AC = CA,  ~BC = CB$, and let $a, b$ and $c$ be positive integers such that $c = ab$. Suppose that  $U$ and $ V$ are unitary operators belonging to the von-Neumann algebra generated by $A$ and $B$, and satisfying $U^a = A$ and $V^b = B$. Then the element $W = UVU^{-1}V^{-1}$ satisfies $UW = WU$, $VW = WV$, and $W^c = C$.
\end{lemma}

\begin{proof} Let $\mathcal A$ be the von Neumann algebra generated by $A$ and $ B$. The group $N$ generated by $A$ and $B$ is isomorphic with the Heisenberg group $\HH$, and so for any $P$ and $Q$ in $N$, $[P,Q] = PQP^{-1}Q^{-1} $ belongs to the center of $N$. It follows that  $[\mathcal A,\mathcal A] \subset \text{cent}(\mathcal A)$ and in particular $W\in\text{cent}(\mathcal A)$. It remains to show that $W^c=C$. To prove this, we proceed by induction on $c = ab$: if $c = 1$ then $a = b = 1$ and there is nothing to prove. Suppose that $c> 1$ and that for any $a', b', c'$ with $a'b' = c'$ and $c' < c$, we have
$$
W^{c'} = U^{a'}V^{b'}U^{-a'}V^{-b'}.
$$
If $a> 1$, then we have
$$ 
W^{(a-1)b} = U^{a-1}V^{b}U^{-a+1}V^{-b}.
$$
 
Observe that $U$ commutes with $V^b U^{-a+1}V^{-b}$: indeed, by definition of $W$, $UV^b = W^bV^b U$ so $UV^{-b} = W^{-b}V^{-b} U$, from which the observation follows. Hence
$$
\begin{aligned}
W^{ab} &= W^{(a-1)b}W^b = \bigl(U^{a-1}V^{b}U^{-a+1}V^{-b}\bigr)\bigl(UV^bU^{-1}V^{-b}\bigr)\\
&= U^{a-1}\bigl(V^{b}U^{-a+1}V^{-b}\bigr)U\bigl(V^bU^{-1}V^{-b}\bigr)\\
&= U^{a-1}U \bigl(V^{b}U^{-a+1}V^{-b}\bigr)\bigl(V^bU^{-1}V^{-b}\bigr)\\
&= U^aV^bU^{-a}V^{-b}.
\end{aligned}
$$
If $a = 1$ then $b > 1$ and the proof is similar.\end{proof}

It is almost immediate that for $\a$ as in the preceding, every representation of $\HH$ has an $\a$-root. 
More generally, we consider the following class of groups that includes $G(\a,\HH)$. Let $n$ be a positive integer, and let $t_1, t_2, \dots , t_n$, and $z_{ij}, 1 \le i,j \le n$ satisfy the relations for all $i, j$ and $k$:
$$
t_it_j=z_{i,j}t_jt_i, \ \ \text{ and }  \ \ \ z_{ij}t_k = t_k z_{ij}.
$$
Observe that the relation $z_{ji} = z_{ij}^{-1}$ follows from the above. The group 
$$
F_n = \langle t_1, t_2, \dots t_n, z_{ij}, 1 \le i,j \le n \rangle
$$
is the free, two-step (discrete) nilpotent group generated by the $n$ elements $t_k, 1 \le k \le n$. 
 
\begin{thm} Define $\a : F_n \rightarrow F_n$ by $\a(t_k) = t_k^{a_k}$ and $\a(z_{ij}) = z_{ij}^{a_ia_j}$ where the $a_k$ are integers. Then every representation of $F_n$ has an $\a$-root. 

\end{thm}

\begin{proof} Let $\rho$ be any representation of $F_n$ acting in $\H$, put $A_k = \rho(t_k)$, $C_{ij} = \rho(z_{ij}), 1 \le i,j, k \le n$ and let $\mathcal A$ be the von-Neumann algebra generated by $\{A_1, \dots , A_n\}$. An argument similar to that of Lemma \ref{first-technical-general version} applied to the group $N$ generated by $\{A_1, \dots , A_n\}$ shows that $[\mathcal A,\mathcal A] \subset  \text{cent}(\mathcal A)$. 
By the Borel functional calculus, for each $k$ we have $U_k\in \mathcal A$ such that $U_k^{a_k} = A_k$. Now for each $i$ and $j$ put $W_{ij} = U_iU_jU_i^{-1}U_j^{-1}$. By the preceding we have $W_{ij}$ is central, and by Lemma \ref{first-technical-general version}, $W_{ij}^{a_ia_j} = C_{i,j}$. Put $T(t_k) = U_k$ and $T(z_{ij}) = W_{ij}, 1 \le i,j,k\le n$. Since 
$$
T(z_{ij}) = T(t_i)T(t_j)T(t_i)^{-1}T(t_j)^{-1}
$$
holds for all $i$ and $j$, then $T$ is a representation of $F_n$. Since 
$$
T(\a(t_k)) = T(t_k^{a_k}) = T(t_k)^{a_k} = A_k = \rho(t_k),
$$
and 
$$
 T(\a(z_{ij})) = T(z_{ij}^{a_ia_j}) = T(z_{ij})^{a_ia_j} = C_{ij} =  \rho(z_{ij}),
 $$
 then $T\circ \a = \rho$. \end{proof}

The following are two examples of representations of $G(a,a,\HH)$ where $\HH$ is the simply connected Heisenberg group. 

\begin{eg} {\rm Let $\pi$ be the representation of $G(2,2,\HH)$ acting in $L^2(\R^2)$ by $t_1\mapsto e^{2\pi i \l}I$, $t_2\mapsto M$ and $t_3\rightarrow T$ where $I$ is the identity operator, and $M$ and $T$ are the operators on $\mathcal H= L^2(\R^2)$ given by 
$$ Mf(\l, t)= e^{-2\pi i \l t}f(\l, t),~ ~ Tf(\l, t)= f(\l, t-1). $$
Now define $\pi(u)f(\l,t) = f(4\l,2^{-1}t) 2^{3/2}$. 
The systems $\pi(\G)\psi$ are Fourier transform of wavelet systems of multiplicity one subspaces of $L^2(\HH)$, and large classes of Parseval wavelet frames have been found in our earlier work  \cite{CM1}. 

}

\end{eg}

\begin{eg}[Shearlet system]\label{shear}
{\rm 
Let  $\pi$  be the representation of $G(a,a,\HH)$ given by $u\mapsto D$, $t_1\mapsto T_1$, $t_2\mapsto T_2$, and $t_3\mapsto M$, where $D, T_1, T_2, M$ are the unitary operators on $ L^2(\R^2)$ defined by
\begin{align}\notag
Df(x)&= a^{-3/2}f(a^{-2}x_1, a^{-1}x_2)  &Mf(x)&= f(x_1-x_2, x_2)\\\notag
 T_1f(x)&=f(x_1-1, x_2) & T_2f(x)&=f(x_1, x_2-1). 
\end{align}}
Systems of this form are well-studied; see for example \cite{DG},\cite{ EL}, and \cite{ GL}. 
\end{eg} 

\noindent
{\bf Remark.} Lemma \ref{first-technical-general version} can be used to prove that for other nilpotent groups $\G_0$, every representation has an $\a$-root. For example, let
$$\G_0 = \langle t_1, t_2, t_3, t_4, t_5 : t_5 t_4 = t_4t_5t_2, t_5t_3 = t_3t_5 t_1,  t_it_j = t_jt_i, 1 \le i,j \le 4\rangle;$$
 
$\G_0$ is the integer lattice in a two-step simply-connected group Lie group whose Lie algebra has basis $\{X_1, X_2, \dots , X_5\}$ with $[X_5,X_4] =  X_2$ and $[X_5, X_3] = X_1$, $[X_i,X_j] = 0, 1 \le i,j\le 4$. Let $a$ and $b$ be integers and define $\a : \G_0 \rightarrow \G_0$ by $\a(t_5) = t_5^a$, $\a(t_k) = t_k^b, k = 3, 4$ and for $k = 1, 2, \a(t_k) = t_k^{ab}$. By application of Lemma  \ref{first-technical-general version} to $\{\pi(t_5), \pi(\t_3), \pi(t_1)\}$ and $\{\pi(t_5), \pi(\t_4), \pi(t_2)\}$, we find that $\pi$ has an $\a$-root. One example of $\pi$ is the following.
Let  $\pi : G \rightarrow \mathcal U\bigl(L^2(\R^4)\bigr)$  be given by $u\mapsto D$, $t_k\mapsto T_k, k = 1, 2, 3, 4$, and $t_5\mapsto M$, where $T_k$ is the translation operator $ T_kf(x)=f(x_1, \dots , x_k - 1, \dots x_4)$ and $D, M$ are defined by
$$
\begin{aligned}
Df(x)&= a^{-3/2}f((ab)^{-1}x_1, (ab)^{-1}x_2, a^{-1}x_3, a^{-1}x_4)  \\
Mf(x)&= f(x_1-x_3, x_2 - x_4, x_3, x_4).
\end{aligned}
$$

\vspace{.5cm}

{\bf Acknowledgment.} Part of this paper was written while the second author was visiting   Department of Mathematics and Computer Science, St. Louis University, USA, Summer 2010. She is thankful for their support and hospitality.

\vspace{1cm}

 Bradley Currey , 
 {Department of Mathematics and Computer
Science, Saint Louis University, St. Louis, MO 63103}\\
 { \footnotesize{E-mail address: \texttt{{ curreybn@slu.edu}}}\\}

  Azita Mayeli, 
   {Mathematics Department, New York City College of Technology, CUNY,   Brookly, 11201, USA }\\ 
  \footnotesize{E-mail address: \texttt{{amayeli@citytech.cuny.edu}}}\\

\end{document}